\documentclass[12pt]{amsart}
\usepackage{wasysym}
\usepackage{amsmath}
\usepackage{amssymb}
\usepackage{amsthm}
\usepackage{caption}
\usepackage{subcaption}
\usepackage{tikz}
\usetikzlibrary{matrix,arrows,backgrounds,shapes.misc,shapes.geometric,patterns,calc,positioning}
\usetikzlibrary{calc,shapes}
\usepackage{wrapfig}
\usepackage{epsfig}  
\usepackage[margin=1in]{geometry}
\usepackage{color}	 %
\usepackage{multicol}
\input{xy}
\xyoption{poly}
\xyoption{2cell}
\xyoption{all}

\usepackage{lscape}

\usetikzlibrary{calc,shapes}
\usetikzlibrary{matrix,arrows,backgrounds,shapes.misc,shapes.geometric,patterns,calc,positioning}
\usetikzlibrary{calc,shapes}

\newcommand{\calg}{\mathcal{G}}
\newcommand{\calh}{\mathcal{H}}
\newcommand{\call}{\mathcal{L}}

\makeatletter
\def\Ddots{\mathinner{\mkern1mu\raise\p@
		\vbox{\kern7\p@\hbox{.}}\mkern2mu
		\raise4\p@\hbox{.}\mkern2mu\raise7\p@\hbox{.}\mkern1mu}}
\makeatother

\newtheorem{thm}{Theorem}[section]

\newtheorem{lem}[thm]{Lemma}

\theoremstyle{defn}
\newtheorem{definition}[thm]{Definition}
\newtheorem{example}[thm]{Example}

\theoremstyle{remark}

\numberwithin{equation}{section}

\DeclareMathSizes{5}{3}{2}{2}

\newcommand{\sL}{\mathcal{L}}

\newcommand\restr[2]{{
		\left.\kern-\nulldelimiterspace 
		#1 
		\vphantom{\big|} 
		\right|_{#2} 
	}}

\begin{document}
		\title{F-polynomial Formula from Continued Fractions}
		\author{Michelle Rabideau}\thanks{The author was supported by the NSF-CAREER grant  DMS-1254567, and by the University of Connecticut.}
		\address{Department of Mathematics,
			University of Connecticut,
		Storrs, Connecticut 06269, USA}
		\email{Michelle.Rabideau@uconn.edu}

		
		
		
		
		%
		%
		%
		\begin{abstract} 
			For cluster algebras from surfaces, there is a known formula for cluster variables and $F$-polynomials in terms of the perfect matchings of snake graphs. If the cluster algebra has trivial coefficients, there is also a known formula for cluster variables in terms of continued fractions. In this paper, we extend this result to cluster algebras with principal coefficients by producing a formula for the $F$-polynomials in terms of continued fractions.
			
		\end{abstract}

		\maketitle
		
		\setcounter{tocdepth}{1}

		
		\section{Introduction}
		
		Cluster algebras were introduced in 2002 by Fomin and Zelevinsky \cite{FZ1}. These cluster algebras are commutative rings  $\mathcal{A}(\mathbf{x},\mathbf{y}, Q)$ that depend on the initial cluster variables $\mathbf{x} = (x_1, \dots, x_d)$, cluster coefficients $\mathbf{y} = (y_1,\dots, y_d)$ and a quiver $Q$. The combinatorial structure of these rings is based on a recursive operation called mutation.  In this paper, we focus on cluster algebras with principal coefficients because a cluster variable in an arbitrary cluster algebra can be computed from a cluster variable in a corresponding cluster algebra with principal coefficients \cite{FZ4}.

		Each iteration of a mutation yields a cluster variable in the form of a Laurent polynomial in terms of the initial cluster variables $x_1,\dots, x_d$. The coefficients of the terms in this Laurent polynomial are monomials in the cluster coefficients. For each cluster variable there exists an $F$-polynomial obtained from the cluster variable by setting each initial cluster variable, $x_i$, equal to one. It was proved in \cite{FZ4} that $F$-polynomials are polynomials in the cluster coefficients.
		
		For cluster algebras of surface type, there is a bijection between the arcs $\gamma$ on the surface and the cluster variables of the cluster algebra \cite{FST}. From an arc $\gamma$ in a triangulated surface, one can construct a labeled snake graph \cite{MSW}. For our purposes, the labeling of the snake graph is unimportant, however each numbered tile $i$ corresponds to the principal coefficient with that subscript, $y_i$. This snake graph is useful because there is an expansion formula that gives the cluster variable $x_\gamma$ as a sum over the perfect matchings of the snake graph \cite{MSW}. Since the cluster variable $x_\gamma$ can be obtained from the perfect matchings,  so can the $F$-polynomial of $x_\gamma$. 
		
		Theorem 4.1 in \cite{CS4} tells us there is a bijection between positive continued fractions and snake graphs. In the same paper the authors give a formula for the cluster variable of a labeled snake graph in terms of a continued fraction of Laurent polynomials in $x_1,\dots,x_n$. However, this formula only works for cluster algebras with trivial coefficients. In this paper, we extend this result to cluster algebras with principal coefficients by giving an explicit formula for the $F$-polynomial as a continued fraction of Laurent polynomials in $y_1,\dots,y_n$.

		
			\section{Preliminaries}\label{sect prelim}
			A finite  continued fraction 
			$$[a_1,\dots, a_n]  := a_1 + \cfrac{1}{a_{2} + \cfrac{1}{\ddots + \cfrac{1}{a_n}}}$$\vspace{10pt}\\
			is said to be positive if each $a_i$ is a positive integer. We define $N[a_1,\dots,a_n]$ by the recursion $N[a_1,\dots,a_n] = a_n N[a_1,\dots,a_{n-1}]+N[a_,\dots,a_{n-2}]$ where $N[a_1]=a_1$. It will become useful for us to have the following definition. For a finite, positive continued fraction $[a_1,\dots,a_n]$, define $$\ell_i = \sum_{s=1}^i a_s$$ for $i >0$ and  $\ell_0 =0$.

			We begin by recalling the definition of snake graphs and some of their properties. First, a tile $G$ is a fixed length square in the plane. The sides of a tile are orthogonal or parallel to a fixed basis, so it is natural for one to call the lower edge the south edge, etc. Each tile can either be placed to the right or above the previous tile. A snake graph $\calg$ is a  connected planar graph given by a sequence of tiles $G_i$, where $G_i$ and $G_{i+1}$ share exactly one edge, $e_i$. The edges $e_i$ that are contained in two tiles are called interior edges. For a snake graph made up of tiles $G_1,G_2,\dots G_d$, the interior edges are $e_1,e_2,\dots, e_{d-1}$. All of the other edges are called boundary edges. A subset $P$ of edges of $\calg$ where each vertex in $\calg$ is incident to exactly one edge in $P$ is called a perfect matching. See figure \ref{fig 234minmatch} for an example of a perfect matching on a snake graph with 8 tiles, where each tile $G_i$ is labeled by $i$.   See \cite{CS3} for a more detailed account regarding snake graphs.

		In \cite{CS4} the authors establish a bijection between snake graphs and positive continued fractions. The snake graph obtained from the continued fraction $[a_1,\dots,a_n]$ is denoted by $\calg[a_1,\dots,a_n]$ and consists of tiles $G_1,\dots,G_d$ where $d = (\ell_n)-1$. 
			The snake graph $\calg[a_1,\dots,a_n]$ has $N[a_1,\dots,a_n]$ many perfect matchings \cite{CS4}. For example the snake graph $\calg[2,2]$, has $N[2,2] = 5$ perfect matchings as seen in figure \ref{fig poset22}. These perfect matchings form a poset \cite[Section 5]{MSW2} with a unique minimal and maximal element. These minimal and maximal perfect matchings, $P_-$ and $P_+$ respectively, are the only perfect matchings consisting entirely of boundary edges. For consistency, we adopt the convention that the minimal perfect matching $P_-$ of $\calg[a_1,\dots,a_n]$ is the perfect matching consisting of only boundary edges that also contains the south edge of the first tile. 
			
			For two examples of posets of perfect matchings see figures \ref{fig poset22} and \ref{fig poset4} in example \ref{ex poset}. We can climb the poset of perfect matchings from $P$ to $P'$ by turning a tile. A tile $G$ can be turned if two of its edges are in the perfect matching $P$. By turning the tile, you replace the two edges of $G$ in $P$ with the other two edges of $G$ and in doing so obtain the perfect matching $P'$. The height $y(P)$ of $P$ is defined recursively by $y(P_-) = 1$ and if $P'$ is above $P$ and obtained by turning the tile $G_i$ then $y(P')=y_iy(P)$. The $F$-polynomial of $\calg$ is defined as $F(\calg) = \sum_P y(P)$, where the sum is over all perfect matchings of $\calg$. It was shown in \cite{MSW} that if $\calg$ is the snake graph of a cluster variable in a cluster algebra from a surface then the $F$-polynomial of that cluster variable is $F(\calg)$ (or a specialization of $F(\calg)$ if the corresponding arc crosses a  self-folded triangle).
			
				\begin{example}	\label{ex poset}
				
				\begin{figure}
					\centering
					\begin{minipage}{.6\textwidth}
						\centering
					\scalebox{0.8}{\includegraphics[width=.8\linewidth]{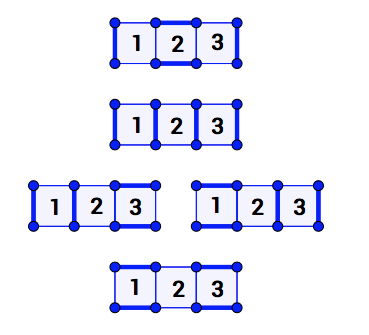}}
						\captionof{figure}{$\calg[2,2]$}
						\label{fig poset22}
					\end{minipage}%
					\begin{minipage}{.4\textwidth}
						\centering
					\scalebox{0.8}	{\includegraphics[width=.5\linewidth]{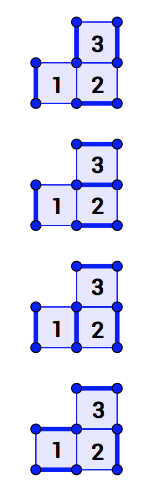}}
						\captionof{figure}{$\calg[4]$}
						\label{fig poset4}
					\end{minipage}
				\end{figure}
				
				In figures \ref{fig poset22} and \ref{fig poset4} below, we give the poset of perfect matchings on the snake graphs $\calg[2,2]$ and $\calg[4]$ respectively, where $P_-$ is at the bottom of the poset and $P_+$ is at the top.   We can use these posets to compute the $F$ polynomials, $F(\calg[2,2]) = 1+y_1+y_3 + y_1y_3 +y_1y_2y_3$ and $F(\calg[4]) = 1+y_1+y_1y_2+y_1y_2y_3$. The snake graph in figure \ref{fig 234minmatch} has $N[2,3,4] = 30$ perfect matchings an the $F$-polynomial is given in example \ref{ex 234} using the main result of this paper. 
				
				 Notice that the poset for $\calg[4]$ is linear, but the poset for $\calg[2,2]$ is not. The snake graph, $\calg[4]$ is a zigzag snake graph, meaning that the placement of tiles alternates between above and to the right of the previous tile. In general, a continued fraction with a single entry yields a zigzag snake graph, $\calg[a_i]$, which always has a linear poset of perfect matchings.
			
		\end{example}
		
				Under the correspondence between continued fractions and snake graphs each $a_i$ in the continued fraction $[a_1,\dots,a_n]$ corresponds to a zigzag subgraph $\calh_i$ of $\calg[a_1,\dots,a_n]$. The subgraph $\calh_i$ is isomorphic to $\calg[a_i]$, but inherits its tile labels from $\calg[a_1,\dots,a_n]$, thus consisting of tiles $G_{(\ell_{i-1})+1},\dots, G_{(\ell_i)-1}$. In figure \ref{fig subgraph} we show the subgraphs $\calh_1,\calh_2,\calh_3$ of the snake graph $\calg[2,3,4]$. Notice that $\calh_3$ is isomorphic to $\calg[4]$, shown in figure \ref{fig poset4}.
			
			The minimal matching, $P_-^i$, of $\calh_i$ is the completion of the matching it inherits from the minimal matching of $\calg[a_1,\dots,a_n]$. Technically speaking, if $P_-$ is the minimal perfect matching of $\calg[a_1,\dots,a_n]$ then $P_-^i$ is given by first collecting all edges in $P_-$ that are also in $\calh_i$, denoted by $P_-|_{\calh_i}$. Then, if this is not already a perfect matching, we must include exactly one other edge as follows.

				\begin{equation*}
				P_-^i = \left\{
				\begin{array}{rl}
				P_-|_{\calh_i} & \text{if } i =1, \text{ or if $i=n$ is even,}\\
				P_-|_{\calh_i} \cup e_{\ell_{i-1}} & \text{if $i$ is odd},\\
				P_-|_{\calh_i} \cup e_{(\ell_i)-1} & \text{if $i$ is even.}\\
				\end{array} \right.
				\end{equation*}\vspace{10pt}
				
				\begin{example}\label{ex 234minmatch}
					In figure \ref{fig 234minmatch}, the minimal matching $P_-$ of $\calg[2,3,4]$ is shown. The snake graph of $[2,3,4]$ has three subgraphs, $\calh_1,\calh_2$ and $\calh_3$. From the minimal perfect matching of $\calg[a_1,\dots,a_n]$ in figure \ref{fig 234minmatch} we obtain $P_-|_{\calh_i}$.  Notice that in figure \ref{fig subgraph} we show these subgraphs with their completed minimal perfect matchings, $P^i_-$.


						\begin{figure}
							\centering
							\begin{minipage}{.5\textwidth}
								\centering
							\scalebox{0.9}	{\includegraphics[width=.8\linewidth]{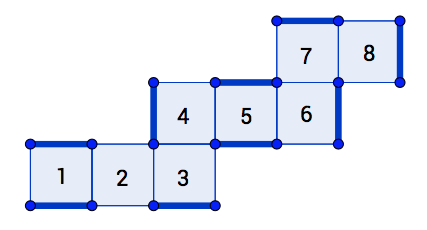}}
								\captionof{figure}{The minimal perfect matching $P_-$ of $\calg[2,3,4]$. } 
							\label{fig 234minmatch}
							\end{minipage}%
							\begin{minipage}{.5\textwidth}
								\centering
								\scalebox{0.9}{\includegraphics[width=.8\linewidth]{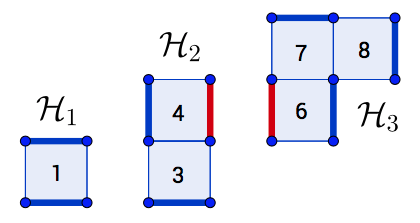}}
								\captionof{figure}{The completed minimal perfect matchings, $P^i_-$ of $\calh_i$ are shown. }
								\label{fig subgraph}
							\end{minipage}
						\end{figure}

					\end{example} 
	In \cite{CS,CS2,CS3} identities in the cluster algebra have been expressed in terms of snake graphs. Equation (\ref{f eqn}) below follows from the grafting with a single edge formula from Theorem 7.3 of \cite{CS2}, where $n\geq2$ and the grafting takes place at tile $G_{\ell_{n-1}}$.
									
		\begin{eqnarray} \label{f eqn} 
		F(\calg[a_1,\dots,a_n])  &= &y_{34}F(\calg[a_1,\dots,a_{n-1}])F(\calh_n) +y_{56}	F(\calg[a_1,\dots,a_{n-2}])  
		\end{eqnarray}\vspace{5pt}

	We define the variables $y_{34}$ and $y_{56}$ as follows, where $y_0 = 1$.\\
						
	\begin{center}
\begin{minipage}{3in}
\begin{equation} \label{yeqn}
  y_{34}= \left\{
    \begin{array}{rl}
    y_{\ell_{n-1}}  & \text{if } n \text{ is odd},\\
      1 & \text{if } n \text{ is even,}\\
    \end{array} \right.
\end{equation}
\end{minipage}
\begin{minipage}{3in}
\begin{equation*}
y_{56}= \left\{
\begin{array}{rl}
 1 & \text{if } n \text{ is odd},\\
 \prod_{j=\ell_{n-2}}^{(\ell_n)-1} y_j& \text{if } n \text{ is even.}\\
\end{array} \right.
\end{equation*} 
\end{minipage}
\end{center} \vspace{10pt}

	For the case when $n=2$, we define the snake graph of an empty continued fraction to be a single edge. There is only one perfect matching of a single edge, which corresponds to the $F$-polynomial 1. Therefore equation (\ref{f eqn}) for the case when $n=2$ is as follows.
	
	\begin{eqnarray} \label{f eqn2} 
	F(\calg[a_1,a_2])  &= &F(\calg[a_1])F(\calh_2) +\prod_{j=0}^{(\ell_2)-1} y_j
	\end{eqnarray}\vspace{5pt}

	Equations (\ref{f eqn}), (\ref{yeqn}) and (\ref{f eqn2}) will be used in the proof of our main theorem. Together they give an equation for the $F$-polynomial of a snake graph $\mathcal{G}[a_1,\dots, a_n]$ based on the $F$-polynomials of certain subgraphs.

		\section{Main Result}\label{sect main}
		
	\begin{definition} \label{defn correction}
		For any continued fraction $[a_1,\dots,a_n]$ with $a_1 >1$, we define an associated continued fraction of Laurent polynomials $[\sL_1,\sL_2,\dots,\sL_n]$, where each $\sL_i = \varphi_iC_i$ and \\ \vspace{10pt}
		
		\begin{minipage}{3in}
			\begin{equation*}
			C_i= \left\{
			\begin{array}{ll}
			\displaystyle  \prod_{j=1}^{\ell_{i-1}} y_j & \text{if } i \text{ is odd},\\[.7cm]
			\displaystyle \prod_{j=1}^{(\ell_i)-1}y_j^{-1} & \text{if } i \text{ is even,}\
			\end{array} \right.
			\end{equation*} \end{minipage}
		\begin{minipage}{3in}
			\begin{equation*}
			\varphi_i= \left\{
			\begin{array}{ll}
			\displaystyle \sum_{k=\ell_{i-1}}^{(\ell_i) -1}\hspace{5pt} \prod_{j=(\ell_{i-1})+1}^{k} y_j  & \text{if } i \text{ is odd},\\[.9cm]
			\displaystyle   \sum_{k=(\ell_{i-1})+1}^{\ell_i}\hspace{5pt} \prod_{j=k}^{(\ell_i)-1} y_j & \text{if } i \text{ is even.}\
			\end{array} \right.
			\end{equation*}
		\end{minipage}\\
	\end{definition}
	
	Next we will prove a lemma that is crucial in the proof of our main theorem. However, to highlight the concept of Lemma \ref{lemma}, we will first give an example. 
	
		\begin{example}
			Again consider the continued fraction $[2,3,4]$ from the previous example \ref{ex 234minmatch}. Here we consider the zigzag snake graphs $\calh_i$ that are subgraphs of $\calg[a_1,\dots,a_n]$. For $\calh_1$, the minimal matching shown in figure \ref{fig subgraph} corresponds to the constant term of the $F$-polynomial, 1, and the maximal matching we obtain by turning tile 1 corresponds to the term $y_1$ in the $F$-polynomial. For $\calh_2$, we have the minimal matching associated to 1, the matching obtained from turning tile 4 associated to  $y_4$ and the maximal matching obtained from turning tile 4 and then tile 3 associated to $y_4y_3$. Together these form the $F$-polynomial $1+y_4+y_4y_3$. The $F$-polynomial of each subgraph $\calh_1,\calh_2$ and $\calh_3$ is given below.
			\begin{center}	$F(\calh_1) = 1+y_1 = \displaystyle \sum_{k=\ell_{0}}^{(\ell_1) -1}\hspace{5pt} \prod_{j=(\ell_{0})+1}^{k} y_j  =\varphi_1$\vspace{10pt} \\
				 $F(\calh_2) = 1+y_4+y_4y_3 = 	\displaystyle   \sum_{k=(\ell_{1})+1}^{\ell_2}\hspace{5pt} \prod_{j=k}^{(\ell_2)-1} y_j  = \varphi_2$ \vspace{10pt} \\
				$F(\calh_3 ) = 1+y_6+y_6y_7+y_6y_7y_8= \displaystyle \sum_{k=\ell_{2}}^{(\ell_3) -1}\hspace{5pt} \prod_{j=(\ell_{2})+1}^{k} y_j  = \varphi_3$\\
			\end{center}
		\end{example}
	\vspace{10pt}

		\begin{lem}
			$F(\calh_i) = \varphi_i$,  \hspace{3pt} for all \hspace{3pt}$1\leq i \leq n$. \label{lemma}
		\end{lem}

		\begin{proof}
			Let $i$ be odd. The tiles $G_{(\ell_{i-1})+1},\dots, G_{(\ell_i)-1}$ make up $\calh_i$. The subgraph $\calh_i$ is the zigzag snake graph $\calg[a_i]$ with the inherited labeling and completed minimal perfect matching $P_-^i$ discussed previously. In this completion, the first tile of the subgraph $G_{(\ell_{i-1})+1}$ can be turned immediately and is the only such tile. If we turn tile $G_{(\ell_{i-1})+1}$, we can then turn the next tile $G_{(\ell_{i-1})+2}$ and so on. Therefore $$F(\calh_n) = 1 + y_{(\ell_{i-1})+1} + y_{(\ell_{i-1})+1}y_{(\ell_{i-1})+2} +\cdots + y_{(\ell_{i-1})+1}\cdots y_{(\ell_i)-1} = \sum_{k=\ell_{i-1}}^{(\ell_i) -1}\hspace{5pt} \prod_{j=(\ell_{i-1})+1}^{k} y_j = \varphi_i.$$
			
			Let $i$ be even. In this case, the very last tile of $\calh_i$, $G_{(\ell_i)-1}$ has two edges in the minimal matching of the completion, $P_-^i$, and can be turned. Therefore in order to determine the $F$-polynomial of $\calh_i$, we must first turn the last tile and work our way down the snake graph. 
			$$F(\calh_i) = 1+ y_{(\ell_i)-1} + y_{(\ell_i)-1}y_{(\ell_i)-2} + \cdots + y_{(\ell_i)-1}\cdots y_{(\ell_{i-1})+1} = \sum_{k=(\ell_{i-1})+1}^{\ell_i}\hspace{5pt} \prod_{j=k}^{(\ell_i)-1} y_j  = \varphi_i \qedhere$$ 
		\end{proof}\vspace{10pt}
		

	We are now ready to prove the main result.
	
	\begin{thm} \label{main thm}
		The F-polynomial associated to the snake graph of the continued fraction $[a_1,\dots,a_n]$ denoted by $F(\calg[a_1,\dots,a_n])$ is given by the equation:\\
		\begin{equation*}
		F(\calg[a_1,\dots,a_n]) = \left\{
		\begin{array}{ll}
		N[\sL_1,\sL_2,\dots,\sL_n] & \text{if } n \text{ is odd},\\[.7cm]
		C_n^{-1} {N[\sL_1,\sL_2,\dots,\sL_n]} & \text{if } n \text{ is even,}\
		\end{array} \right.
		\end{equation*}\\
		where $N[\sL_1,\sL_2,\dots,\sL_n]$ is defined by the recursion $N[\sL_1,\sL_2,\dots,\sL_n] = \sL_n N[\sL_1,\sL_2,\dots,\sL_{n-1}] + N[\sL_1,\sL_2,\dots,\sL_{n-2}] $ and $N[\sL_1] = \sL_1$. \\
	\end{thm} 
	
	\begin{proof}
		Proof by induction. Let $n=1$. In the case where $n$ is odd, by equation (\ref{yeqn}), we have $y_{34} =y_{\ell_{n-1}}$ and $y_{56}=1$. It is clear that $F(\calg[a_1]) = F(\calh_1)$ simply because $\calg[a_1]$ and $\calh_1$ are the same snake graph. Then, by Lemma \ref{lemma} we have $F(\calh_1) = \varphi_1$. Note that $C_1 =  \prod_{j=1}^{0} y_j  = 1$ because it is an empty product. Therefore $\varphi_1 = C_1\varphi_1 = \call_1 = N[\call_1]$. Thus we have shown that $F(\calg[a_1]) =N[\call_1]$.\\
		
		Let $n=2$. In this case we use equation (\ref{f eqn2}) and note that $\prod_{j=0}^{(\ell_2)-1} y_j = C_2^{-1}$.
		
		\begin{eqnarray*}
			F(\calg[a_1,a_2])  &= &F(\calg[a_1])F(\calh_2) +\prod_{j=0}^{(\ell_2)-1} y_j\\ [1em]
		&=& F(\calg[a_1])F(\calh_2) +C_2^{-1} \\	\end{eqnarray*}
	Using Lemma \ref{lemma} and the case $n=1$, we see that the right hand side is equal to $\varphi_2 N[\sL_1]+C_2^{-1}$ and this is equal to the following, where the second equation holds by definition \ref{defn correction} and the last equation holds by the definition of $N[\sL_1,\sL_2]$.
			
		\begin{eqnarray*}	
		F(\calg[a_1,a_2])	&=& C_2^{-1}\left(C_2\varphi_2N[\call_1] +1\right)\\[.5em]
			&=& C_2^{-1}\left(\sL_2N[\call_1] +1\right)\\[.5em]
			&=&C_2^{-1}N[\call_1,\call_2]\\
		\end{eqnarray*}
		
		Now let $n>2$ be odd. Assume that for all $m <n$ our statement holds. In this situation by equation (\ref{yeqn}), $y_{34} =y_{\ell_{n-1}}$ and $y_{56}=1$.  Additionally, we know from equation (\ref{f eqn}) that the F-polynomial of $\calg[a_1,\dots,a_n]$ is given by the following.
		
		\[	F(\calg[a_1,\dots,a_n])  = y_{\ell_{n-1}}F(\calg[a_1,\dots,a_{n-1}])F(\calh_n) +	F(\calg[a_1,\dots,a_{n-2}])\]\\
Applying our inductive step we obtain:

	\[	F(\calg[a_1,\dots,a_n])  = y_{\ell_{n-1}}C_{n-1}^{-1}{N[\sL_1,\sL_2,\dots,\sL_{n-1}]} F(\calh_n) +	{N[\sL_1,\sL_2,\dots,\sL_{n-2}]}\].\\
From here we can apply Lemma \ref{lemma}.
			\[F(\calg[a_1,\dots,a_n])  = y_{\ell_{n-1}} C_{n-1}^{-1}{N[\sL_1,\sL_2,\dots,\sL_{n-1}]} \varphi_n +	{N[\sL_1,\sL_2,\dots,\sL_{n-2}]}\]\\
				Using the fact that $C_n = y_{\ell_{n-1}}C_{n-1}^{-1}$ and  $\sL_n = C_n\varphi_n$ we obtain our desired result as follows.\\
	\begin{eqnarray*}
		F(\calg[a_1,\dots,a_n])  &=& C_n\varphi_nN[\sL_1,\sL_2,\dots,\sL_{n-1}] +	N[\sL_1,\sL_2,\dots,\sL_{n-2}]\\ [.5em]
		&=& \sL_nN[\sL_1,\sL_2,\dots,\sL_{n-1}] +	N[\sL_1,\sL_2,\dots,\sL_{n-2}]\\ [.5em]
		&=&N[\sL_1,\sL_2,\dots,\sL_n] \\
	\end{eqnarray*}
	
	In the case where $n>2$ is even, our argument is very similar. Assume that for all $m <n$ our statement holds. In this case, $y_{34} = 1$ and $y_{56} = \prod_{j=\ell_{n-2}}^{(\ell_n)-1} y_j$. Again, we make the corresponding replacements based on our induction hypothesis.

	\begin{eqnarray*}
		F(\calg[a_1,\dots,a_n])  &=& N[\sL_1,\dots,\sL_{n-1}]F(\calh_n) +\left(\prod_{j=\ell_{n-2}}^{(\ell_n)-1} y_j	\right)C_{n-2}^{-1}N[\sL_1,\dots,\sL_{n-2}])  \\ [.3cm]
	\end{eqnarray*}
	Then we apply Lemma \ref{lemma} and the rest follows similarly to the previous case.
	\begin{eqnarray*}
		F(\calg[a_1,\dots,a_n])  
		&=& N[\sL_1,\sL_2,\dots,\sL_{n-1}] \varphi_n +	\left(\prod_{j=\ell_{n-2}}^{(\ell_n)-1} y_j\right)C_{n-2}^{-1}N[\sL_1,\sL_2,\dots,\sL_{n-2}]\\ [.6em]
		&=& N[\sL_1,\sL_2,\dots,\sL_{n-1}] \varphi_n +C_n^{-1}N[\sL_1,\sL_2,\dots,\sL_{n-2}]\\ [.5em]
	 &=& C_n^{-1}	\left(\varphi_n C_nN[\sL_1,\sL_2,\dots,\sL_{n-1}]  +	N[\sL_1,\sL_2,\dots,\sL_{n-2}]\right)\\ [.5em]
	 &=& C_n^{-1}
		\left(\sL_nN[\sL_1,\sL_2,\dots,\sL_{n-1}]  +	N[\sL_1,\sL_2,\dots,\sL_{n-2}]\right)\\ [.5em]
		&=& C_n^{-1}N[\sL_1,\sL_2,\dots,\sL_{n}] 
	\end{eqnarray*} 
\end{proof}


		\begin{example}
			Consider the continued fraction $[2,3,4,2] = \frac{67}{29}$, then the $F$-polynomial has 67 terms. According to Theorem \ref{main thm}, since the continued fraction has an even number of entries, the F-polynomial of the snake graph $\calg[2,3,4,2]$ is given by $C^{-1}_nN[\call_1,\call_2,\call_3,\call_4]$. 

		\[	\begin{array}{ccccccl}
			\sL_1 &=& 1+y_1
		&\quad &	\sL_3 &=& (1+y_6+y_6y_7+y_6y_7y_8)(y_1y_2y_3y_4y_5)\\ [1em]
			\sL_2 &=& \dfrac{1+y_4+y_3y_4}{y_1y_2y_3y_4}
		&\quad&	\sL_4 &=& \dfrac{1+y_{10}}{y_1y_2y_3y_4y_5y_6y_7y_8y_9y_{10}}
		\end{array}\]
		
			\begin{eqnarray*}
			F(\calg[2,3,4,2])	&=&C^{-1}_nN[\call_1,\call_2,\call_3,\call_4]\\ [1em] &=&\left(\prod_{j=1}^{(\ell_4)-1}y_j\right)(\call_1\call_2\call_3\call_4 + \call_1\call_2+\call_1\call_4+\call_3\call_4 +1)\\ [1em]
				&=& (1+y_1)(1+y_4+y_3y_4)(1+y_6+y_6y_7+y_6y_7y_8)(1+y_{10})y_5 \nonumber \\  &&  +\: (1+y_1)(1+y_4+y_3y_4)y_5y_6y_7y_8y_9y_{10} + (1+y_1)(1+y_{10})\nonumber \\ &&  +\: (1+y_6+y_6y_7+y_6y_7y_8)(1+y_{10})y_1y_2y_3y_4y_5 + y_1y_2y_3y_4y_5y_6y_7y_8y_9y_{10} \\
			\end{eqnarray*}
			
		\end{example}

		\begin{example} \label{ex 234}
			Consider the continued fraction $[2,3,4]$. Since $[2,3,4] = \frac{30}{13}$, the $F$-polynomial has 30 terms.  According to Theorem \ref{main thm}, the F-polynomial of the snake graph $\calg[2,3,4]$ is given by $N[\call_1,\call_2,\call_3]$.\\
			\begin{eqnarray*}
			F(\calg[2,3,4])=	&=& N[\call_1,\call_2,\call_3] \\  [.5em]
				&=&\call_1\call_2\call_3 + \call_1+\call_3\\ [.5em]
&=& (1+y_1)(1+y_4+y_3y_4)(1+y_6+y_6y_7+y_6y_7y_8)y_5
 +(1+y_1) \\&&  +\: (1+y_6+y_6y_7+y_6y_7y_8)y_1y_2y_3y_4y_5\\
			\end{eqnarray*}

		\end{example}

		{}

	\end{document}